\documentclass[10pt,reqno]{article}%{article}
\usepackage{fullpage} % decreases margins
\usepackage{amssymb} % AMS-LaTeX + Theorems + Symbols
\usepackage{graphicx} % to include graphics files
\usepackage{amsmath} % AMS-LaTeX + Theorems + Symbols
\usepackage{amsthm,amssymb,latexsym} % AMS-LaTeX + Theorems + Symbols
\usepackage{amstext, amsfonts,amsbsy} % not sure why I need it
\usepackage{enumerate} % easy customization of enumeration
\usepackage{upgreek} % for \upalpha, etc.
\usepackage{color} % adds color
\usepackage{accents} % for \undertilde
\usepackage{mathtools} % for dcases, etc.
\usepackage{float} % forcing figure floats
\usepackage{tikz}
\usetikzlibrary{matrix,graphs,arrows,positioning,calc,decorations.markings,shapes.symbols}
\usepackage[pdftex,bookmarks,colorlinks,breaklinks]{hyperref}  % PDF hyperlinks, with coloured links
\definecolor{dullmagenta}{rgb}{0.4,0,0.4}   % #660066
\definecolor{darkblue}{rgb}{0,0,0.4}
\hypersetup{linkcolor=red,citecolor=blue,filecolor=dullmagenta,urlcolor=darkblue} % coloured links
\usepackage{eucal}

\newtheorem{theorem}{Theorem}%[section]

\newtheorem{lemma}[theorem]{Lemma}

\newcommand{\dPain}[1]{\text{P}\left(\mathrm{#1}\right)}

\theoremstyle{definition}

\theoremstyle{remark}
\newtheorem{remark}[theorem]{Remark}

\numberwithin{equation}{section}
\begin{document}
{\noindent\Large\bf Recurrence coefficients for the semiclassical Laguerrre weight and d-$\dPain{A_{2}^{(1)}/E_{6}^{(1)}}$ equations}
\medskip
\begin{flushleft}
\textbf{Siqi Chen}\\
School of Mathematics and Statistics, Qilu University of Technology (Shandong Academy of Sciences)
Jinan 250353, China\\
E-mail: \href{mailto:chen_siqi0301@163.com}{\texttt{chen\_siqi0301@163.com}}\\[5pt]

\textbf{Mengkun Zhu}\\
School of Mathematics and Statistics, Qilu University of Technology (Shandong Academy of Sciences)
Jinan 250353, China\\
E-mail: \href{mailto:zmk@qlu.edu.cn}{\texttt{zmk@qlu.edu.cn}}\\[5pt]
\end{flushleft}
\begin{abstract}
In this paper, we use Sakai's geometric framework to explore the profound interconnection between recurrence coefficients of the semiclassical Laguerre weight
$w(x)=x^{\lambda}\mathrm{e}^{-x^2+sx}$, $x\in\mathbb{R}^+$, $\lambda>-1$, $s\in\mathbb{R}$,
and Painlev\'e equations.
Specifically, we introduce a new transformation for the expressions obtained by Filipuk et al. in their analysis of ladder operators for semiclassical Laguerre polynomials, thereby deriving a recurrence relation. Subsequently, we establish a correspondence between this recurrence relation and a class of  d-$\dPain{A_{2}^{(1)}/E_{6}^{(1)}}$ equations.
\end{abstract}
\section{Introduction} % (fold)
\label{sec:introduction}
Orthogonal polynomials play a fundamental role in various branches of mathematics and mathematical physics, including random matrix theory, approximation theory, numerical analysis, and so on. In recent years, the relationship between recurrence coefficients of semiclassical orthogonal polynomials and \mbox{solutions} to discrete or differential Painlev\'{e} equations has attracted significant attention. For example, Magnus \cite{Mag:1995:PTDEFTRCOSOP} studied the weight function
\begin{equation*}
w(x)=\mathrm{e}^{x^3/3+sx},\quad x^3<0, \quad s\in\mathbb{R},
\end{equation*}
and established a connection between the recurrence coefficients and the $\mathrm{P_{II}}$ equation. In \cite{CI:2010:PAASLSIHRMEI},
Chen and Its investigated the singularly perturbed Laguerre weight
\begin{equation*}
w(x)=x^{\lambda}\mathrm{e}^{-x-s/x}, \quad x\geq0, \quad \lambda\in\mathbb{R}^+, \quad s\in\mathbb{R}^+.
\end{equation*}
Their findings revealed that the diagonal recurrence coefficient of monic orthogonal polynomials satisfies a particular $\mathrm{P_{III'}}$ equation. Clarkson and Jordaan \cite{CJ:2014:TRBSLPATFPE} showed that the recurrence coefficients of the semiclassical Laguerre weight
\begin{equation}\label{eq-wf}
w(x):=w(x;s)=x^{\lambda}\mathrm{e}^{-x^2+sx}, \quad x\in\mathbb{R}^{+}, \quad \lambda > -1, \quad s\in\mathbb{R},
\end{equation}
satisfy differential equations that are related to $\mathrm{P_{IV}}$ equation. See also \cite{BV:2010:DPEFRCOSLP,CJK:2016:AGFW,FVZ:2012:TRCOSLPATFPE,VA:2018:OPAPE}.

In \cite{Sak:2001:RSAWARSGPE}, Sakai proposed a classification scheme for both continuous and discrete Painlev\'e equations, based on the geometric theory of Painlev\'e equations. He gave a complete classification of all possible configuration spaces for discrete Painlev\'e dynamics, demonstrating that each such configuration space is a family of specific rational algebraic surfaces, known as generalized Halphen surfaces.

In this paper, based on Sakai's work, we also focus on the semiclassical Laguerre weight \eqref{eq-wf}, while another semiclassical Laguerre weight $w(x)=x^{\lambda}\mathrm{e}^{-N(x+c(x^2-x))}$ is discussed in \cite{DzhFilSto:2022:DEFTRCOSOPATRTTPEVGA}. Consider a sequence of monic polynomials $P_{n}(x;s)=x^{n}+\mathbf{p}(n,s)x^{n-1}+\cdots$ orthogonal with respect to the weight \eqref{eq-wf}, i.e.,
\begin{equation*}
\int_{0}^{\infty}P_{j}(x;s)P_{k}(x;s)x^{\lambda}\mathrm{e}^{-x^2+sx}\mathrm{d}x=\delta_{j,k}h_{j}(s),
\end{equation*}
where $h_{j}(s)$ denotes the $L^2$-norm of squared of $P_{j}(x;s)$. These monic orthogonal polynomials satisfy the three-term recurrence relation
\begin{equation*}
xP_{n}(x;s)=P_{n+1}(x;s)+\alpha_{n}(s)P_{n}(x;s)+\beta_{n}(s)P_{n-1}(x;s),
\end{equation*}
with the initial condition $P_{0}(x;s)=1$, $P_{-1}(x;s)=0$.

In \cite{BV:2010:DPEFRCOSLP}, it was shown that the recurrence coefficients satisfy the discrete system
\begin{equation}\label{eq:bv-evol}
	\left\{ \begin{aligned}
		\widetilde{x}_{n-1} \widetilde{x}_{n} &= \frac{\widetilde{y}_n+n+\frac{1}{2}\lambda}{\widetilde{y}_{n}^2-\frac{1}{4}\lambda^2},\\
		\widetilde{y}_{n} + \widetilde{y}_{n+1} &=\frac{1}{\widetilde{x}_n}\left(\frac{s}{\sqrt{2}}-\frac{1}{\widetilde{x}_{n}}\right),
	\end{aligned}\right.
\end{equation}
where
\begin{equation*}
\widetilde{x}_{n}(s)=\frac{\sqrt{2}}{s-2\alpha_{n}(s)}, \qquad \widetilde{y}_{n}(s)=2\beta_{n}(s)-n-\frac{1}{2}\lambda,
\end{equation*}
and the system \eqref{eq:bv-evol} can be derived from an asymmetric d-$\mathrm{P_{IV}}$ equation via a limiting process.

Another way to studying the recurrence coefficients of orthogonal polynomials $P_n(x;s)$ is to use
\begin{alignat*}{2}
	\left(\frac{\textrm{d}}{\textrm{d}x}+B_n(x;s)\right)P_n(x;s)&= \beta_n(s) A_n(x;s) P_{n-1}(x;s),&\quad &\text{lowering operator}\\
	\left(\frac{\textrm{d}}{\textrm{d}x}-B_{n}(x;s)-v'(x)\right)P_{n-1}(x;s)&= -A_{n-1}(x;s)P_{n}(x;s),&\quad &\text{raising operator}
\end{alignat*}
where $A_n(x;s)$ and $B_n(x;s)$ are parameterized by the functions $R_{n}(s)$ and $r_{n}(s)$,
\begin{equation}\label{eq-AB}
	xA_n(x;s)=2x+R_n(s),\qquad xB_n(x;s)=r_n(s),
\end{equation}
here,
\begin{align*}
	R_n(s)&= \frac{\lambda}{h_n(s)}\int_{0}^{\infty}P_{n}^2(y;s)y^{\lambda-1}\mathrm{e}^{-y^2+sy}\mathrm{d}y,\\
    r_n(s)&=
\frac{\lambda}{h_{n-1}(s)}\int_{0}^{\infty}P_{n}(y;s)P_{n-1}(y;s)y^{\lambda-1}\mathrm{e}^{-y^2+sy}\mathrm{d}y .
\end{align*}

The functions $A_n(x)$ and $B_n(x)$ also satisfy the following modified compatibility conditions,
\begin{align}
B_{n+1}(x)+B_n(x)&=\left(x-\alpha_n(s)\right)A_n(x)-v'(x),\tag{$S_1$}\label{$S_1$}\\
1+\left(x-\alpha_n(s)\right)\left(B_{n+1}(x)-B_{n}(x)\right)&=\beta_{n+1}(s)A_{n+1}(x)-\beta_{n}(s)A_{n-1}(x),\tag{$S_2$}\\
B_n^2(x)+v'(x)B_n(x)+\sum_{j=0}^{n-1}A_{j}(x)&=\beta_n(s)A_n(x)A_{n-1}(x),\tag{$S_2'$}\label{$S_2'$}
\end{align}
where $v(x)=-\ln w(x)$. In \cite{FVZ:2012:TRCOSLPATFPE}, Filipuk et al. obtained the following by substituting equation \eqref{eq-AB} into \eqref{$S_1$} and \eqref{$S_2'$} and comparing the coefficients of the same powers of $x$,
\begin{align}
R_{n}(s)&=2\alpha_{n}(s)-s,\label{eq-1}\\
r_{n}(s)+r_{n-1}(s)&=\lambda-\alpha_{n}(s)R_{n}(s),\label{eq-2}\\
r_{n}(s)&=2\beta_{n}(s)-n,\label{eq-3}\\
r_{n}^2(s)-\lambda r_{n}(s)&=\beta_{n}(s)R_{n-1}(s)R_{n}(s),\label{eq-4}
\end{align}
where \eqref{eq-1}-\eqref{eq-4} correspond to (21), (22), (25) and (27) in \cite{FVZ:2012:TRCOSLPATFPE}.

According to \eqref{eq-1}-\eqref{eq-4}, we introduce a new different transformation $x_n(s)$, $y_n(s)$ via
\begin{equation*}
x_{n}(s):=\frac{1}{R_{n-1}(s)}, \qquad y_{n}(s):=-r_n(s),
\end{equation*}
which yields
\begin{equation}\label{eq:xyn-evol}
	\left\{ \begin{aligned}
		x_{n} x_{n+1} &= \frac{n-y_n}{2y_n^2+2\lambda y_n},\\
		y_{n} + y_{n-1} &= -\frac{2\lambda x_n^2-sx_n-1}{2x_n^2}.
	\end{aligned}\right.
\end{equation}
This recurrence relation mirrors the structure of (25) in \cite{Gra:2014:DPEAIP}, known to generate discrete Painlev\'{e} equations. By systematically applying the reduction methodology outlined in \cite{DzhFilSto:2020:RCDOPHWDPE}, we demonstrate that this recurrence relation \eqref{eq:xyn-evol} is a discrete Painlev\'{e} equation that is equivalent to the standard
example in the d-$\dPain{A_{2}^{(1)}/E_{6}^{(1)}}$ family. Our main result is as follows.
\begin{theorem}\label{thm:coordinate-change}
	The recurrence \eqref{eq:xyn-evol} is equivalent to the standard discrete Painlev\'e
	equation~\eqref{eq:dPE6-KNY} 
\begin{equation*}
	\overline{q} + q = p - t - \frac{a_2}{p},\qquad
	p + \underline{p} = q + t + \frac{a_1}{q},
\end{equation*}
    written in
	\cite{KajNouYam:2017:GAOPE}. This equivalence is achieved
	via the following change of variables:
	\begin{equation}\label{eq:xy2qp}
			x(q,p) = \frac{q}{\sqrt{2}(a_1-qp)},\quad
			y(q,p) = qp-a_{1},\quad
            s(t)=\sqrt{2}t.
	\end{equation}
	The inverse change of variables is given by	
	\begin{equation}\label{eq:qp2xy}
		q(x,y) = -\sqrt{2}xy,\quad
		p(x,y) = \frac{n-y}{\sqrt{2}xy},\quad
        t(s)=\frac{s}{\sqrt{2}}.
	\end{equation}
	The relationship between the semiclassical Laguerre weight recurrence parameters and the root variables of discrete Painlev\'e equations is given by
	\begin{equation}
		a_{0} =  1-\lambda, \quad a_{1} = -n,\quad
			a_{2} =  n + \lambda.
	\end{equation}
\end{theorem}

% section introduction (end)
\section{Discrete Painlev\'e Equations in the \lowercase{d}-$\dPain{A_{2}^{(1)}/E_{6}^{(1)}}$
family} % (fold)
\label{sec:discrete_painlev_e_equations_in_the_lowercase_d_dpain_a__2_1_e__6_1_family}
In this section, to ensure the self-contained nature of this paper, we have summarized the fundamental facts about the geometry of the $E_{6}^{(1)}$-family of Sakai surfaces and the standard discrete Painlev\'{e} equation associated with this surface family. For the standard example, we use $(q,p)$-coordinates and adopt the standard surface root basis illustrated in Figure~\ref{fig:d-roots-e61} and the standard symmetry root basis shown in Figure~\ref{fig:a-roots-a2}, as outlined in the reference \cite{KajNouYam:2017:GAOPE}.
\begin{figure}[ht]
\begin{equation}\label{eq:d-roots-e61}
	\raisebox{-32.1pt}{\begin{tikzpicture}[
			elt/.style={circle,draw=black!100,thick, inner sep=0pt,minimum size=2mm}]
		\path 	(0,2) 	node 	(d0) [elt, label={[xshift=-10pt, yshift = -10 pt] $\delta_{0}$} ] {}
		        (-2,0) node 	(d1) [elt, label={[xshift=0pt, yshift = -20 pt] $\delta_{1}$} ] {}
		        (-1,0) 	node  	(d2) [elt, label={[xshift=0pt, yshift = -20 pt] $\delta_{2}$} ] {}
		        ( 0,0) 	node  	(d3) [elt, label={[xshift=0pt, yshift = -20 pt] $\delta_{3}$} ] {}
		        ( 1,0) 	node  	(d4) [elt, label={[xshift=0pt, yshift = -20 pt] $\delta_{4}$} ] {}
		        ( 2,0) node 	(d5) [elt, label={[xshift=0pt, yshift = -20 pt] $\delta_{5}$} ] {}
                ( 0,1) node 	(d6) [elt, label={[xshift=-10pt, yshift = -10 pt] $\delta_{6}$}]
{};
		\draw [black,line width=1pt ] (d0) -- (d6) -- (d3) (d1) -- (d2) -- (d3)  (d3) -- (d4) -- (d5);
	\end{tikzpicture}} \qquad
\begin{aligned}
&\delta_0=\mathcal{E}_{7}-\mathcal{E}_{8},
\\
&\delta_1=\mathcal{E}_{1}-\mathcal{E}_{2},\quad
\\
&\delta_2=\mathcal{H}_{q}-\mathcal{E}_{1}-\mathcal{E}_{5},\quad
\end{aligned}
\begin{aligned}
&\delta_3=\mathcal{E}_{5}-\mathcal{E}_{6},
\\
&\delta_4=\mathcal{H}_{p}-\mathcal{E}_{3}-\mathcal{E}_{5},\quad
\\
&\delta_5=\mathcal{E}_{3}-\mathcal{E}_{4},
\\
&\delta_6=\mathcal{E}_{6}-\mathcal{E}_{7}.
\end{aligned}
\end{equation}
	\caption{The standard surface root basis for the $E_6^{(1)}$ surface sub-lattice.}
	\label{fig:d-roots-e61}
\end{figure}
\begin{figure}[ht]
\begin{equation}\label{eq:a-roots-a2}			
	\raisebox{-32.1pt}{\begin{tikzpicture}[
			elt/.style={circle,draw=black!100,thick, inner sep=0pt,minimum size=2mm}]
		\path 	(0,1) 	node 	(a0) [elt, label={[xshift=-10pt, yshift = -10 pt] $\alpha_{0}$} ] {}
		        (-1.15,-1) node 	(a1) [elt, label={[xshift=-10pt, yshift = -10 pt] $\alpha_{1}$} ] {}
		        (1.15,-1) node  	(a2) [elt, label={[xshift=10pt, yshift = -10 pt] $\alpha_{2}$} ] {};
		\draw [black,line width=1pt ] (a0) -- (a1) -- (a2)  -- (a0);
	\end{tikzpicture}} \qquad
			\begin{aligned}
&\alpha_0=\mathcal{H}_q+\mathcal{H}_p-\mathcal{E}_5-\mathcal{E}_6-\mathcal{E}_7-\mathcal{E}_8,
\\
&\alpha_1=\mathcal{H}_q-\mathcal{E}_3-\mathcal{E}_4,\quad
\\
&\alpha_2=\mathcal{H}_p-\mathcal{E}_1-\mathcal{E}_2,
\\
&\delta=\alpha_0+\alpha_1+\alpha_2.
\end{aligned}
\end{equation}
	\caption{The standard symmetry root basis for the $A_2^{(1)}$ symmetry sub-lattice.}
	\label{fig:a-roots-a2}	
\end{figure}
\subsection{The Point Configuration} % (fold)
\label{sub:the_point_configuration}
The Picard lattice of a rational algebraic surface $\mathcal{X}$, which is obtained by blowing up eight points on $\mathbb{P}^{1}\times\mathbb{P}^1$, is generated by the following classes,
\begin{equation*}
\operatorname{Pic}(\mathcal{X}) = \operatorname{Span}_{\mathbb{Z}}\{\mathcal{H}_{q},\mathcal{H}_{p},\mathcal{E}_{1},\ldots, \mathcal{E}_{8}\},
\end{equation*}
and this lattice is equipped with the symmetric bilinear product (the intersection form), which is defined on the generators by
$\mathcal{H}_{q}\bullet \mathcal{H}_{q} = \mathcal{H}_{p}\bullet \mathcal{H}_{p} = \mathcal{H}_{q}\bullet \mathcal{E}_{i} =
\mathcal{H}_{p}\bullet \mathcal{E}_{j} = 0,$ $\mathcal{H}_{q}\bullet \mathcal{H}_{p} = 1,$ and  $\mathcal{E}_{i}\bullet \mathcal{E}_{j} = - \delta_{ij}.$ Within this lattice, the anti-canonical divisor class is given by
$
- \mathcal{K}_{\mathcal{X}} = 2 \mathcal{H}_{q} + 2 \mathcal{H}_{p} - \mathcal{E}_{1}
	 - \mathcal{E}_{2} - \mathcal{E}_{3} - \mathcal{E}_{4} - \mathcal{E}_{5} - \mathcal{E}_{6} - \mathcal{E}_{7} - \mathcal{E}_{8}.
$
For the $E_{6}^{(1)}$ surface, this class should decompose into irreducible components each with self-intersection $-2$, specifically as
\begin{equation*}
- \mathcal{K}_{\mathcal{X}} = \delta = \delta_{0} + \delta_{1} + 2 \delta_{2} + 3 \delta_{3} + 2\delta_{4} + \delta_{5} + 2\delta_{6},
\end{equation*}
and this decomposition is given by the choice of the surface root basis shown on Figure~\ref{fig:d-roots-e61}.

We now proceed to describe the corresponding point configuration. Let $Q=1/q$ and $P=1/p$ denote the coordinates at infinity. By using the M\"{o}bius group action, we can arrange the points to be $p_{1}(\infty,0)$, $p_{3}(0,\infty)$, $p_{5}(\infty,\infty)$, and the
only remaining gauge action is the rescaling of the $q$-coordinate (to be utilized later for normalizing the root variables). We then get the point configuration shown on Figure~\ref{fig:surface-e6}. It is worth noting that one reason for adopting this point normalization is that they are located on the polar divisor of the standard symplectic form $\omega=\mathrm{d}q\wedge\mathrm{d}p$.
\begin{figure}[H]
	\begin{tikzpicture}[>=stealth,basept/.style={circle, draw=red!100, fill=red!100, thick, inner sep=0pt,minimum size=1.2mm}]
		\begin{scope}[xshift = 0cm]
			\draw [black, line width = 1pt]  	(4,0) 	-- (-0.5,0) 	node [left] {$H_{p}$} node[pos=0, right] {$p=0$};
			\draw [black, line width = 1pt] 	(4,2.5) -- (-0.5,2.5)	node [left] {$H_{p}$} node[pos=0, right] {$p=\infty$};
			\draw [black, line width = 1pt] 	(0,3) -- (0,-0.5)		node [below] {$H_{q}$} node[pos=0, above, xshift=-7pt] {$q=0$};
			\draw [black, line width = 1pt] 	(3,3) -- (3,-0.5)		node [below] {$H_{q}$} node[pos=0, above, xshift=7pt] {$q=\infty$};
			
			\node (p1) at (3,0) 	[basept,label={[xshift=10pt, yshift = -15 pt] $p_{1}$}] {};
			\node (p2) at (2.5,0.5) 	[basept,label={[xshift=-10pt, yshift = -5 pt] $p_{2}$}] {};
			\node (p3) at (0,2.5) [basept,label={[xshift=-10pt, yshift =0 pt] $p_{3}$}] {};
			\node (p4) at (0.5,2.0) [basept,label={[xshift=10pt, yshift =-15 pt] $p_{4}$}] {};
			\node (p5) at (3,2.5) 	[basept,label={[xshift=-10pt, yshift =0 pt] $p_{5}$}] {};
			\node (p6) at (3.5,2) 	[basept,label={[xshift=0pt, yshift =-18 pt] $p_{6}$}] {};
			\node (p7) at (4.1,2)		[basept,label={[xshift=0pt, yshift =-18 pt] $p_{7}$}] {};
			\node (p8) at (4.7,2) [basept,label={[xshift=0pt, yshift =-18 pt] $p_{8}$}] {};
			\draw [line width = 0.8pt, ->] (p8) edge (p7) (p7) edge (p6) (p6) edge (p5) (p2) edge (p1) (p4) edge (p3);
		\end{scope}

		\draw [->] (7,1.5)--(5,1.5) node[pos=0.5, below] {$\operatorname{Bl}_{p_{1}\cdots p_{8}}$};

		\begin{scope}[xshift = 9cm]
			\draw [blue, line width = 1pt] 	(2.5,2.5) -- (0.5,2.5)	node [left] {$H_{p} - E_{3} - E_{5}$};
			\draw [red, line width = 1pt] 	(0,2) -- (0,-0.5)		node [below] {$H_{q}-E_{3}$};
			\draw [red, line width = 1pt ] (2.5,0) -- (-0.5,0)        node [left] {$H_{p} - E_{1}$};
			\draw [blue, line width = 1pt] 	(3,2)--(3,0.5)          node [right] {$H_{q} - E_{1}-E_{5}$};

			\draw [blue, line width = 1 pt] (1.1,2.6) -- (-0.1,1.4) node [left] {$E_{3}-E_{4}$};
            \draw [blue, line width = 1 pt] (3.1,1.1) -- (1.9,-0.1) node [below] {$E_{1}-E_{2}$};
			\draw [red, line width = 1 pt] (0.2,2.3) -- (0.8,1.7) node [right]{$E_{4}$};
            \draw [red, line width = 1 pt] (2.8,0.2) -- (2.2,0.8) node [left]{$E_{2}$};
			\draw [blue, line width = 1 pt] (3.1,1.6) -- (2.1,2.6) node [above] {$E_{5}-E_{6}$};
            \draw [blue, line width = 1 pt] (2.3,1.8) -- (3.5,3) node [above right] {$E_{6}-E_{7}$};
            \draw [blue, line width = 1 pt] (3.0,2.8) -- (3.8,2) node [right]{$E_{7}-E_{8}$};
            \draw [red, line width = 1 pt] (3.2,2.2) -- (3.6,2.6) node [ right] {$E_{8}$};
		\end{scope}
	\end{tikzpicture}
	\caption{The model Sakai surface for the d-$\mathrm{P}\left(A_2^{(1)}/E_6^{(1)}\right)$ example.}
	\label{fig:surface-e6}
\end{figure}
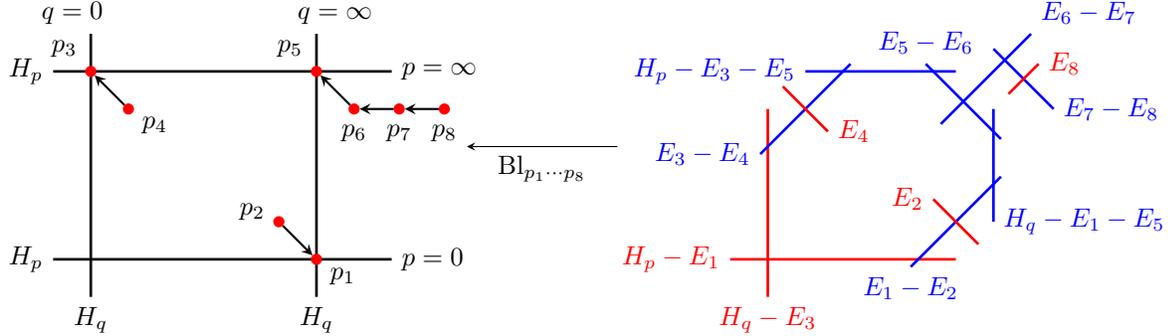
The parameterization of this point configuration in terms of root variables $a_{0},a_{1},a_{2}$ normalized by $a_{0}+a_{1}+a_{2}=1$ is given in \cite{KajNouYam:2017:GAOPE},
\begin{equation*}
p_{12}\left(\frac{1}{\varepsilon},-a_{2}\varepsilon\right)_{2}, \quad p_{34}\left(a_1\varepsilon,\frac{1}{\varepsilon}\right)_{2},\quad
p_{5678}\left(\frac{1}{\varepsilon},\frac{1}{\varepsilon}+t+(a_{1}+a_{2}-1)\varepsilon\right)_{4},
\end{equation*}
or, explicitly,
\begin{alignat}{2}
	&p_{1}\left(Q=\frac{1}{q}=0,p=0\right)&\leftarrow
    &p_{2}\left(u_1=\frac{1}{q}=0,v_1=qp=-a_2\right),\notag\\ &p_{3}\left(q=0,P=\frac{1}{p}=0\right)&\leftarrow
    &p_{4}\left(U_3=qp=a_1, V_3=\frac{1}{p}=0\right),\notag\\
	&p_{5}\left(Q=\frac{1}{q}=0,P=\frac{1}{p}=0\right)&\leftarrow
	&p_{6}\left(u_5=\frac{1}{q}=0,v_5=\frac{q}{p}=1\right)\\
	&&\leftarrow &p_{7}\left(u_6=\frac{1}{q}=0,v_6=\frac{q^2-qp}{p}=-t\right)\notag\\
	&&\leftarrow &p_{8}\left(u_7=\frac{1}{q}=0,v_7=\frac{q^3-q^2p+tqp}{p}=t^2+a_0\right).\notag
\end{alignat}
% subsection the_point_configuration (end)
\subsection{The Extended Affine Weyl Symmetry Group} % (fold)
\label{sub:the_extended_affine_weyl_symmetry_group}
Recall that, the algebraic source of Painlev\'e dynamics stems from the birational representation of the extended affine Weyl symmetry group $\widetilde{W}\left(A_{2}^{(1)}\right)=\mathrm{Aut}\left(A_{2}^{(1)}\right)\ltimes W\left(A_{2}^{(1)}\right)$.

The affine Weyl group $W\left(A_{2}^{(1)}\right)$ is defined via generators $w_{i} = w_{\alpha_{i}}$ and relations that are encoded in the diagram on Figure~\ref{fig:a-roots-a2},
\begin{equation*}
	W\left(A_{2}^{(1)}\right) = W\left(\raisebox{-20pt}{\begin{tikzpicture}[
			elt/.style={circle,draw=black!100,thick, inner sep=0pt,minimum size=1.5mm}]
		\path 	(0,0.5) 	node 	(a0) [elt, label={[xshift=-10pt, yshift = -10 pt] $\alpha_{0}$} ] {}
		        (-0.575,-0.5) node 	(a1) [elt, label={[xshift=-10pt, yshift = -10 pt] $\alpha_{1}$} ] {}
		        (0.575,-0.5) 	node 	(a2) [elt, label={[xshift=10pt, yshift = -10 pt] $\alpha_{2}$} ] {};
		\draw [black,line width=1pt ] (a0) -- (a1) -- (a2)-- (a0);
	\end{tikzpicture}} \right)
	=
	\left\langle w_{0},w_{1}, w_{2}\ \left|\
	\begin{alignedat}{2}
    w_{i}^{2} = e,\quad  w_{i}\circ w_{j} &= w_{j}\circ w_{i}& &\text{ when
   				\raisebox{-0.08in}{\begin{tikzpicture}[
   							elt/.style={circle,draw=black!100,thick, inner sep=0pt,minimum size=1.5mm}]
   						\path   ( 0,0) 	node  	(ai) [elt] {}
   						        ( 0.5,0) 	node  	(aj) [elt] {};
   						\draw [black] (ai)  (aj);
   							\node at ($(ai.south) + (0,-0.2)$) 	{$\alpha_{i}$};
   							\node at ($(aj.south) + (0,-0.2)$)  {$\alpha_{j}$};
   							\end{tikzpicture}}}\\
    w_{i}\circ w_{j}\circ w_{i} &= w_{j}\circ w_{i}\circ w_{j}& &\text{ when
   				\raisebox{-0.17in}{\begin{tikzpicture}[
   							elt/.style={circle,draw=black!100,thick, inner sep=0pt,minimum size=1.5mm}]
   						\path   ( 0,0) 	node  	(ai) [elt] {}
   						        ( 0.5,0) 	node  	(aj) [elt] {};
   						\draw [black] (ai) -- (aj);
   							\node at ($(ai.south) + (0,-0.2)$) 	{$\alpha_{i}$};
   							\node at ($(aj.south) + (0,-0.2)$)  {$\alpha_{j}$};
   							\end{tikzpicture}}}
	\end{alignedat}\right.\right\rangle.
\end{equation*}
In our setting, this group is represented by actions on $\operatorname{Pic}(\mathcal{X})$ induced by reflections in the roots $\alpha_{i}$,
\begin{equation}\label{eq:root-refl}
	w_{i}(\mathcal{C}) = w_{\alpha_{i}}(\mathcal{C}) = \mathcal{C} - 2
	\frac{\mathcal{C}\bullet \alpha_{i}}{\alpha_{i}\bullet \alpha_{i}}\alpha_{i}
	= \mathcal{C} + \left(\mathcal{C}\bullet \alpha_{i}\right) \alpha_{i},\qquad \mathcal{C}\in \operatorname{Pic(\mathcal{X})}.
\end{equation}
Next, we need to extend this group by the Dynkin diagram automorphism group $\mathrm{Aut}\left(A_{2}^{(1)}\right)$, which is isomorphic to the dihedral group $\mathbb{D}_3$ (the symmetry group of a triangle). This group is generated by two reflections (here we use the standard cycle
notations for permutations),
	\begin{equation}
		\sigma_1=(\alpha_1\alpha_2)=(\delta_1\delta_5)(\delta_2\delta_4),\qquad
		\sigma_2=(\alpha_0\alpha_2)=(\delta_0\delta_1)(\delta_2\delta_6).
	\end{equation}
Then $\sigma_{i}$ act on the $\operatorname{Pic}(\mathcal{X})$ as
	\begin{equation*}
			\sigma_1=w_{\mathcal{E}_1-\mathcal{E}_3}\circ w_{\mathcal{E}_2-\mathcal{E}_4}\circ w_{\mathcal{H}_q-\mathcal{H}_p},\qquad
			\sigma_2=w_{\mathcal{E}_1-\mathcal{E}_7}\circ w_{\mathcal{E}_2-\mathcal{E}_8}\circ w_{\mathcal{H}_q-\mathcal{E}_5-\mathcal{E}_6}.		
	\end{equation*}
\begin{lemma}\cite{TD:2025:OPFTGWWAJADPE}\label{thm:bir-weyl-aut-a2}
	The generators of the extended affine Weyl group $\widetilde{W}\left(A_{2}^{(1)}\right)$ act to transform an initial point configuration
\begin{equation*}
\left(\begin{matrix} a_{0}, & a_{1}, & a_{2} \end{matrix}\ ; \ t \ ;
        \begin{matrix} q \\ p \end{matrix}\right)
\end{equation*}
as follows:
	\begin{alignat}{2}
		w_{0}&:
		\left(\begin{matrix} a_{0}, & a_{1}, & a_{2} \end{matrix}\ ; \ t \ ;
        \begin{matrix} q \\ p \end{matrix}\right)
		&&\mapsto
        \left(\begin{matrix} -a_{0}, & a_{1}+a_{0}, & a_{2}+a_{0} \end{matrix}\ ; \ t \ ;
        \begin{matrix} q-\frac{a_0}{q-p+t} \\ p-\frac{a_0}{q-p+t} \end{matrix}\right), \label{w0}\\
		w_{1}&:
        \left(\begin{matrix} a_{0}, & a_{1}, & a_{2} \end{matrix}\ ; \ t \ ;
        \begin{matrix} q \\ p \end{matrix}\right)
		&&\mapsto
        \left(\begin{matrix} a_{0}+a_{1}, & -a_{1}, & a_{2}+a_{1} \end{matrix}\ ; \ t \ ;
        \begin{matrix} q \\ p-\frac{a_1}{q} \end{matrix}\right), \label{w1}\\
		w_{2}&:
		\left(\begin{matrix} a_{0}, & a_{1}, & a_{2} \end{matrix}\ ; \ t \ ;
        \begin{matrix} q \\ p \end{matrix}\right)
		&&\mapsto
        \left(\begin{matrix} a_{0}+a_{2}, & a_{1}+a_{2}, & -a_{2} \end{matrix}\ ; \ t \ ;
        \begin{matrix} q+\frac{a_2}{p} \\ p \end{matrix}\right),\\
        \sigma_{1}&:
		\left(\begin{matrix} a_{0}, & a_{1}, & a_{2} \end{matrix}\ ; \ t \ ;
        \begin{matrix} q \\ p \end{matrix}\right)
		&&\mapsto
        \left(\begin{matrix} -a_{0}, & -a_{2}, & -a_{1} \end{matrix}\ ; \ t \ ;
        \begin{matrix} -p \\ -q \end{matrix}\right), \\
        \sigma_{2}&:
        \left(\begin{matrix} a_{0}, & a_{1}, & a_{2} \end{matrix}\ ; \ t \ ;
        \begin{matrix} q \\ p \end{matrix}\right)
		&&\mapsto
        \left(\begin{matrix} -a_{2}, & -a_{1}, & -a_{0} \end{matrix}\ ; \ t \ ;
        \begin{matrix} q \\ q-p+t \end{matrix}\right).
	\end{alignat}	
\end{lemma}
\begin{proof}
This proof is standard, for details, refer to \cite{DzhFilSto:2020:RCDOPHWDPE} and \cite{DzhTak:2018:OSAOSGTODPE}. Here, we only briefly outline the computation of \eqref{w1}. The reflection $w_1$ in the root $\alpha_{1}=\mathcal{H}_q-\mathcal{E}_3-\mathcal{E}_4$
acts on the $\operatorname{Pic}(\mathcal{X})$ by
\begin{equation*}
w_{1}(\mathcal{H}_q)=\mathcal{H}_q,\quad w_{1}(\mathcal{H}_p)=\mathcal{H}_q+\mathcal{H}_p-\mathcal{E}_3-\mathcal{E}_4,
\end{equation*}
and
\begin{equation*}
w_{1}(\mathcal{E}_3)=\mathcal{H}_q-\mathcal{E}_4,\quad
w_{1}(\mathcal{E}_4)=\mathcal{H}_q-\mathcal{E}_3,\quad
w_{1}(\mathcal{E}_i)=\mathcal{E}_i,\quad i=1,2,5,6,7,8.
\end{equation*}
Thus, we are looking for a mapping $w_1$ which, in the affine chart $(q,p)$, is defined by a formula $w_1(q,p)=(\overline{q},\overline{p})$, so that
\begin{equation*}
w_{1}^*(\mathcal{H}_{\overline{q}})=\mathcal{H}_q,\quad w_{1}^*(\mathcal{H}_{\overline{p}})=\mathcal{H}_q+\mathcal{H}_p-\mathcal{E}_3-\mathcal{E}_4.
\end{equation*}
Hence, up to M\"{o}bius transformations, $\overline{q}$ coincides with $q$, and $\overline{p}$ is a coordinate on a pencil of $(1,1)$-curves passing through the degeneration cascade $p_3(0,\infty)\leftarrow p_4(U_3=a_1,V_3=0)$, then
\begin{equation*}
|H_{\overline{p}}|=\{Aqp+Bq+Cp+D=0\}=\{A(qp-a_1)+Bq=0\}.
\end{equation*}
Considering M\"{o}bius transformations, we obtain
\begin{equation*}
\overline{q}=\frac{Aq+B}{Cq+D}, \quad
\overline{p}=\frac{Kq+L(qp-a_1)}{Mq+N(qp-a_1)},
\end{equation*}
where $A,\ldots,N$ are constants to be determined. We can also get the root variables change as $\overline{a}_0=a_0+a_1$, $\overline{a}_1=-a_1$ and $\overline{a}_2=a_2+a_1$.

Since $w_1(\mathcal{E}_1-\mathcal{E}_2)=\mathcal{E}_1-\mathcal{E}_2$, $(\overline
{q},\overline{p})(\infty,0)=(\infty,0)$, we find that $C=0,K=0,D=1$ and $L=1$. Similarly, from $w_1(\mathcal{E}_3-\mathcal{E}_4)=\mathcal{E}_3-\mathcal{E}_4$, it follows that $(\overline
{q},\overline{p})(0,\infty)=(0,\infty)$, hence $B=0$. From $w_1(\mathcal{E}_5-\mathcal{E}_6)=\mathcal{E}_5-\mathcal{E}_6$, we further find that $(\overline
{q},\overline{p})(\infty,\infty)=(\infty,\infty)$, so $N=0$. Thus,
\begin{equation*}
\overline{q}=Aq, \quad
\overline{p}=\frac{qp-a_1}{Mq},
\end{equation*}
from $w_1(\mathcal{E}_2)=\mathcal{E}_2$, we derive $\frac{A}{M}=1$, from $w_1(\mathcal{E}_6-\mathcal{E}_7)=\mathcal{E}_6-\mathcal{E}_7$, we get $AM=1$. Finally, from $w_1(\mathcal{E}_7-\mathcal{E}_8)=\mathcal{E}_7-\mathcal{E}_8$, we obtain $A=M=1$.
\end{proof}
% subsection the_extended_affine_weyl_symmetry_group (end)

\subsection{Discrete Painlev\'e equations on the $E_{6}^{(1)}$ surface} % (fold)
\label{sub:Standard_discrete_d_dpain_a__2_1_e__6_1_equation}
In \cite{KajNouYam:2017:GAOPE}, the standard example of a discrete Painlev\'{e} equation on the $E_{6}^{(1)}$-surface is given in Section 8.1.18 equation (8.25) as
\begin{equation}\label{eq:dPE6-KNY}
	\overline{q} + q = p - t - \frac{a_2}{p},\qquad
	p + \underline{p} = q + t + \frac{a_1}{q},
\end{equation}
with the root variables evolution and normalization as follows
\begin{equation}\label{eq:dPE6-rv-evol}
	\overline{a}_{0} = a_{0}, \quad \overline{a}_{1} = a_{1}-1, \quad \overline{a}_{2} = a_{2} + 1,\qquad
	 a_{0} + a_{1} + a_{2} = 1.
\end{equation}
From the evolution of the root variables (\ref{eq:dPE6-rv-evol}) we can immediately see that the corresponding translation on the root lattice is
\begin{equation}
	\varphi_{*}: \upalpha =  \langle \alpha_{0}, \alpha_{1}, \alpha_{2}  \rangle
	\mapsto \varphi_{*}(\upalpha) = \upalpha + \langle 0,1,-1 \rangle \delta.
\end{equation}
Using the standard techniques, see \cite{DzhTak:2018:OSAOSGTODPE} for a detailed example, we get the following decomposition of $\varphi$ in terms of the generators of $\widetilde{W}\left(A_{2}^{(1)}\right)$:
\begin{equation}
	\varphi = \sigma_{1}\sigma_{2}w_{0}w_{2}.
\end{equation}
\begin{remark}
With the following relabeling of the root basis,
\begin{equation*}
b_{0}=a_{2}, \quad b_{2}=a_{1}, \quad b_{1}=a_{0},
\end{equation*}
and the substitutions
\begin{equation*}
f=-q, \quad g=p,
\end{equation*}
equation \eqref{eq:dPE6-KNY} is found to coincide with equations (2.39)-(2.40) in \cite{Sak:2007:PDPEATLF}, i.e.
\begin{equation}\label{eq:dPE6-SK7}
	f + \overline{f} = t - g + \frac{b_0}{g},\qquad
	g + \overline{g} = t - \overline{f} - \frac{b_2-1}{\overline{f}}.
\end{equation}
This system of equations is referred to as the d-$\mathrm{P_{II}}$ equation in Section 7 of \cite{Sak:2001:RSAWARSGPE}.
\end{remark}
% subsubsection kajiwara_noumi_yamada_or_d_pain_iii_equation_on_the_e__6_1_surface (end)

% subsection standard_discrete_d_dpain_a__2_1_e__6_1_equations (end)

% section discrete_painlev_e_equations_in_the_lowercase_d_dpain_a__2_1_e__6_1_family (end)
\section{The Identification Procedure} % (fold)
\label{sec:the_identification_procedure}
In this Section, we follow the reduction procedures introduced in \cite{DzhFilSto:2020:RCDOPHWDPE} to establish the correspondence between the recurrence relation \eqref{eq:xyn-evol} and the standard example \eqref{eq:dPE6-KNY}.
\subsection{The Singularity Structure} % (fold)
\label{sub:the_singularity_structure}
In the geometric analysis of discrete Painlev\'{e} equations, the first step is to understand the singularity structure of the system. The recurrence relation \eqref{eq:xyn-evol} induces two fundamental mappings, the \emph{forward mapping} $\psi_{1}^{(n)}: (x_{n},y_{n})\mapsto (x_{n+1},y_{n})$ and the
\emph{backward mapping} $\psi_{2}^{(n)}:(x_{n},y_{n})\mapsto (x_{n},y_{n-1})$. In this paper, we focus on the composed mapping $\psi^{(n)} = \left(\psi_{2}^{(n+1)}\right)^{-1} \circ \psi_{1}^{(n)}: (x_{n},y_{n})\mapsto (x_{n+1},y_{n+1})$. We put
$x:=x_{n}$, $\overline{x}: = x_{n+1}$,
$y:=y_{n}$, $\overline{y}: = y_{n+1}$  and omit the index
$n$ in the mapping notation. The map $\psi:(x,y)\mapsto (\overline{x},\overline{y})$ then becomes
\begin{equation}\label{eq:fwd}
	\left\{
	\begin{aligned}
		\overline{x} &= \frac{n-y}{2xy(y+\lambda)},\\
		\overline{y} &= -\frac{(y+\lambda)\left(n^2-n(2+sx)y+(1+sx-2\lambda x^2)y^2-2x^2y^3\right)}{(y-n)^2}.
	\end{aligned}
	\right.
\end{equation}
To compactify the affine complex plane $\mathbb{C}\times\mathbb{C}$ into $\mathbb{P}^1\times\mathbb{P}^1$, we introduce three supplementary coordinate charts $(X,y)$, $(x,Y)$ and $(X,Y)$, where $X=1/x$, $Y=1/y$. By examining the coordinate where both the numerator and the denominator of the mapping vanish, we immediately see the following base points,
\begin{equation*}
q_{1}(0,n), \quad q_{2}(\infty,-\lambda), \quad q_{3}(\infty,0), \quad q_{4}(0,\infty).
\end{equation*}
At each of these base points, we perform the blowup procedure, see, e.g., \cite{Sha:2013:BAG1}, which entails introducing two new local coordinate charts, $(u_i,v_i)$ and $(U_i,V_i)$, in the neighborhood of the base point $q_i(x_i,y_i)$. The variable transformations are defined as
\begin{equation*}
x=x_i+u_i=x_i+U_iV_i, \quad y=y_i+u_iv_i=y_i+V_i.
\end{equation*}
The coordinates $v_i=1/U_i$ parameterize all possible slopes of lines through $q_i$, and so this variable change `separates' all curves passing through $q_i$ based on their slopes. The blowup procedure induces a bijection on the punctured neighborhood of $q_i$, replacing $q_i$ with a projective line $\mathbb{P}^1$ (\emph{exceptional divisor} $F_i$), locally defined by $u_i=0$ and $V_i=0$ in the blowup charts. Extending the mapping to these charts via coordinate substitutions may reveal new base points on $F_i$ (where $u_i=V_i=0$). For discrete Painlev\'{e} case, iterative blowups finitely resolve all base points, then the following lemma is established.

\begin{lemma}
	\label{lem:base-pt-semi-laguerre}
The base points of the mapping \eqref{eq:fwd}  %in the singularity cascades
are
\begin{alignat}{2}\label{eq-base-pt-semi-laguerre}
	&q_{1}(x=0,y=n),&\quad &\quad\notag\\
    &q_{2}\left(X=\frac{1}{x}=0,y=-\lambda\right)&,\quad &q_{3}\left(X=\frac{1}{x}=0,y=0\right),\notag\\
	&q_{4}\left(x=0,Y = \frac{1}{y}= 0\right)&\leftarrow
	&q_{5}\left(u_4=x=0,v_4=\frac{1}{xy}=0\right)\notag\\
	&&\leftarrow &q_{6}\left(u_5=x=0,v_5=\frac{1}{x^2y}=2\right)\\
	&&\leftarrow &q_{7}\left(u_6=x=0,v_6=\frac{1-2x^2y}{x^3y}=-2s\right)\notag\\
    &&\leftarrow &q_{8}\left(u_7=x=0,v_7=\frac{1-2x^2y+2sx^3y}{x^4y}=2(s^2+2(n+\lambda-1))\right).\notag
\end{alignat}
Considering the inverse mapping does not add any new base points.
\end{lemma}

After the blowup of all eight base points $q_i$, we obtain a (family of) rational algebraic surfaces parameterized by $\lambda$, $s$ and $n$ (the coordinates of the base points), denoted as $\mathcal{X}=\mathcal{X}_{\mathbf{b}}$ with $\mathbf{b}=\{\lambda,s,n\}$.

%subsection the_singularity_structure (end)

\subsection{The Induced Mapping on $\operatorname{Pic}(\mathcal{X})$} % (fold)
\label{sub:the_induced_mapping_on_operatorname_pic_mathcal_x}
In the identification procedure, the next step involves computing the induced mapping on the Picard lattice. For the product space $\mathbb{P}^1\times\mathbb{P}^1$, its Picard lattice is generated by the linear equivalence classes of the coordinate lines. Specifically, we have $\operatorname{Pic}(\mathbb{P}^{1} \times \mathbb{P}^{1}) = \operatorname{Span}_{\mathbb{Z}}\{\mathcal{H}_{x},\mathcal{H}_{y}\}$, where $\mathcal{H}_{x} = [H_{x=a}]$ denotes the class of a \emph{vertical} line and  $\mathcal{H}_{y} = [H_{y=b}]$ denotes the class of a \emph{horizontal} line on $\mathbb{P}^{1}\times \mathbb{P}^{1}$. Each blowup procedure at a base point $q_{i}$ adds the class $\mathcal{F}_{i} = [F_{i}]$
of the \emph{exceptional divisor} of the blowup, expanding the Picard lattice to
\begin{equation*}
\operatorname{Pic}(\mathcal{X}) = \operatorname{Span}_{\mathbb{Z}}\{\mathcal{H}_{x},\mathcal{H}_{y},\mathcal{F}_{1},\ldots, \mathcal{F}_{8}\}.
\end{equation*}
$\operatorname{Pic}(\mathcal{X})$ is equipped with the symmetric bilinear \emph{intersection form} given by
\begin{equation}\label{eq:int-form}
\mathcal{H}_{x}\bullet \mathcal{H}_{x} = \mathcal{H}_{y}\bullet \mathcal{H}_{y} = \mathcal{H}_{x}\bullet \mathcal{F}_{i} =
\mathcal{H}_{y}\bullet \mathcal{F}_{j} = 0,\qquad \mathcal{H}_{x}\bullet \mathcal{H}_{y} = 1,\qquad  \mathcal{F}_{i}\bullet \mathcal{F}_{j} = - \delta_{ij}	
\end{equation}
on the generators, and then extended by the linearity.

The mapping $\psi$ induces a linear map on $\operatorname{Pic}(\mathcal{X})$, note that $\operatorname{Pic}(\mathcal{X})$ and $\operatorname{Pic}(\overline{\mathcal{X}})$ are clearly canonically isomorphic, so we often use the notation $\operatorname{Pic}(\mathcal{X})$. We denote by $\overline{F}_{i}$ the class of the exceptional divisor obtained by the blowup at
$\overline{q}_{i} = \psi(q_{i})$ and use the notation $\mathcal{F}_{i\cdots j}=\mathcal{F}_{i}+ \cdots + \mathcal{F}_{j}$. This computation follows standard computation detailed in \cite{DzhTak:2018:OSAOSGTODPE,DzhFilSto:2020:RCDOPHWDPE}, so we only summarize the result here.
\begin{lemma}\label{lem:dyn}
	The action of the mapping $\psi_{*}: \operatorname{Pic}(\mathcal{X})\to \operatorname{Pic}(\overline{\mathcal{X}})$
	is given by
	\begin{align*}
\begin{aligned}
&\mathcal{H}_x\mapsto4\overline{\mathcal{H}}_x+2\overline{\mathcal{H}}_y-\overline{\mathcal{F}}_{23}-2\overline{\mathcal{F}}_{456}-\overline{\mathcal{F}}_{78}, \\
&\mathcal{F}_1\mapsto2\overline{\mathcal{H}}_x+\overline{\mathcal{H}}_y-\overline{\mathcal{F}}_{45678},
\\
&\mathcal{F}_2\mapsto2\overline{\mathcal{H}}_x+\overline{\mathcal{H}}_y-\overline{\mathcal{F}}_{34567},
\\
&\mathcal{F}_3\mapsto2\overline{\mathcal{H}}_x+\overline{\mathcal{H}}_y-\overline{\mathcal{F}}_{24567},
\\
&\mathcal{F}_4\mapsto\overline{\mathcal{H}}_x+\overline{\mathcal{H}}_y-\overline{\mathcal{F}}_{456},
\end{aligned}
&&
\begin{aligned}
&\mathcal{H}_y\mapsto2\overline{\mathcal{H}}_x+\overline{\mathcal{H}}_y-\overline{\mathcal{F}}_{4567}, \\
&\mathcal{F}_5\mapsto\overline{\mathcal{H}}_x-\overline{\mathcal{F}}_{6},
\\
&\mathcal{F}_6\mapsto\overline{\mathcal{H}}_x-\overline{\mathcal{F}}_{5},
\\
&\mathcal{F}_7\mapsto\overline{\mathcal{H}}_x-\overline{\mathcal{F}}_{4},
\\
&\mathcal{F}_8\mapsto\overline{\mathcal{F}}_{1}.
\end{aligned}
\end{align*}
The evolution of parameters (and hence, the base points) is given by
	$\mathbf{b}=\{\lambda, s, n\}\mapsto \overline{\mathbf{b}}=\{\lambda,s,n+1\}$.
\end{lemma}

% subsection the_induced_mapping_on_operatorname_pic_mathcal_x (end)

\subsection{The Surface Type} % (fold)
\label{sub:the_surface_type}
Given that the mapping is fully regularized via eight blowups, it naturally fits into the discrete Painlev\'{e} equations framework. To determine the algebraic surface type, we analyze the configuration of irreducible components of the bi-degree $(2,2)$ curve $\Gamma$ that contains the base points $q_{1},\ldots,q_{8}$ of the mapping. For generic parameters, the proper transform of $\Gamma$ under these blowups is the unique \emph{anti-canonical divisor} $-K_{\mathcal{X}}$, which corresponds to the polar divisor of a symplectic form $\omega$ and serves as a critical invariant for algebraic surface classification. The projection mapping
\begin{equation*}
	\eta: \mathcal{X}_{\mathbf{b}} = \operatorname{Bl}_{q_{1}\cdots q_{8}}(\mathbb{P}^{1} \times \mathbb{P}^{1}) \to \mathbb{P}^{1} \times \mathbb{P}^{1},
\end{equation*}
formally establishes the birational equivalence between the singular initial space and the regularized blown-up surface $\mathcal{X}_{\mathbf{b}}$, where the eight blowups resolve base-point singularities to embed the mapping within the discrete Painlev\'{e} framework.
\begin{lemma}
The base points $q_{1},\ldots,q_{8}$ of the mapping \eqref{eq:fwd} are situated on the bi-quadratic curve $\Gamma$, which is defined in the affine chart by the equation $x=0$. The homogeneous equation of $\Gamma$ is $x^{0}x^{1}(y^{1})^2= 0$, where $x = x^{0}/x^{1}$ and $y=y^{0}/y^{1}$, confirming that $\Gamma$ is indeed a bi-quadratic curve. It is important to note that certain points exhibit infinitely-close degeneration cascades. The irreducible components $d_{i}$ of the \emph{proper transform} $-K_{\mathcal{X}}$ of $\Gamma$,
\begin{equation*}
	-K_{\mathcal{X}} = 2 H_{x} + 2 H_{y} - F_{1} -\cdots - F_{8} =d_{0} + d_{1} + 2 d_{2} + 3 d_{3} + 2d_{4} + d_{5} + 2d_{6},
\end{equation*}
are given by
\begin{equation}
\begin{aligned}
d_{0} &= F_{7} - F_{8}, \quad d_{1} = H_{x} - F_{2} - F_{3}, \quad d_{2} = H_{y} - F_{4} - F_{5}, \quad d_{3} = F_{5} - F_{6}, \\
d_{4} &= F_{4} - F_{5}, \quad d_{5} = H_{x} - F_{1} - F_{4}, \quad d_{6} = F_{6} - F_{7},
\end{aligned}
\end{equation}
they define the \emph{surface root basis} $\delta_{1},\ldots, \delta_{6}$ (where $\delta_i=[d_i]$) of $-2$-classes in $\operatorname{Pic}(\mathcal{X})$
whose configuration is described by the Dynkin diagram of type $E_{6}^{(1)}$:
\begin{figure}[H]
\begin{equation}\label{eq:d-roots-slw}
	\raisebox{-32.1pt}{\begin{tikzpicture}[
			elt/.style={circle,draw=black!100,thick, inner sep=0pt,minimum size=2mm}]
		\path 	(0,2) 	node 	(d0) [elt, label={[xshift=-10pt, yshift = -10 pt] $\delta_{0}$} ] {}
		        (-2,0) node 	(d1) [elt, label={[xshift=0pt, yshift = -20 pt] $\delta_{1}$} ] {}
		        (-1,0) 	node  	(d2) [elt, label={[xshift=0pt, yshift = -20 pt] $\delta_{2}$} ] {}
		        ( 0,0) 	node  	(d3) [elt, label={[xshift=0pt, yshift = -20 pt] $\delta_{3}$} ] {}
		        ( 1,0) 	node  	(d4) [elt, label={[xshift=0pt, yshift = -20 pt] $\delta_{4}$} ] {}
		        ( 2,0) node 	(d5) [elt, label={[xshift=0pt, yshift = -20 pt] $\delta_{5}$} ] {}
                ( 0,1) node 	(d6) [elt, label={[xshift=-10pt, yshift = -10 pt] $\delta_{6}$}]
{};
		\draw [black,line width=1pt ] (d0) -- (d6) -- (d3) (d1) -- (d2) -- (d3)  (d3) -- (d4) -- (d5);
	\end{tikzpicture}} \qquad
\begin{aligned}
&\delta_0=\mathcal{F}_7-\mathcal{F}_8,
\\
&\delta_1=\mathcal{H}_x-\mathcal{F}_2-\mathcal{F}_3,\quad
\\
&\delta_2=\mathcal{H}_y-\mathcal{F}_4-\mathcal{F}_5,
\end{aligned}
\begin{aligned}
&\delta_3=\mathcal{F}_5-\mathcal{F}_6,
\\
&\delta_4=\mathcal{F}_4-\mathcal{F}_5,\quad
\\
&\delta_5=\mathcal{H}_x-\mathcal{F}_1-\mathcal{F}_4,
\\
&\delta_6=\mathcal{F}_6-\mathcal{F}_7.
\end{aligned}
\end{equation}
	\caption{The surface root basis for the semiclassical Laguerre weight recurrence.}
	\label{fig:d-roots-slw}
\end{figure}
\end{lemma}
The generalized Cartan matrix of affine type $E_{6}^{(1)}$ \cite{Kac:1990:IDLA} is
\begin{equation}
\delta_{i}\bullet\delta_{j}=\begin{pmatrix}
                            -2 & 0  & 0  & 0  & 0  & 0  & 1 \\
                             0 & -2 & 1  & 0  & 0  & 0  & 0 \\
                             0 & 1  & -2 & 1  & 0  & 0  & 0 \\
                             0 & 0  & 1  & -2 & 1  & 0  & 1 \\
                             0 & 0  & 0  & 1  & -2 & 1  & 0 \\
                             0 & 0  & 0  & 0  & 1  & -2 & 0 \\
                             1 & 0  & 0  & 1  & 0  & 0  & -2 \\
                            \end{pmatrix}.
\end{equation}

In Figure \ref{fig:surface-semi-laguerre}, we display the final stage of the blowup process and the resulting $E_{6}^{(1)}$ surface. Consequently, our recurrence relation is classified within the d-$\dPain{A_{2}^{(1)}/E_{6}^{(1)}}$ family, characterized by the symmetry group $\widetilde{W}\left(A_{2}^{(1)}\right)$. The details of the standard d-$\dPain{A_{2}^{(1)}/E_{6}^{(1)}}$ point configuration, root bases for the surface, symmetry sub-lattices, and other relevant data are documented in Appendix, following the conventions established in \cite{KajNouYam:2017:GAOPE}.
\begin{figure}[ht]
	\begin{tikzpicture}[>=stealth,basept/.style={circle, draw=red!100, fill=red!100, thick, inner sep=0pt,minimum size=1.2mm}]
		\begin{scope}[xshift = 0cm]
			\draw [black, line width = 1pt]  	(4,0) 	-- (-0.5,0) 	node [left] {$H_{y}$} node[pos=0, right] {$y=0$};
			\draw [black, line width = 1pt] 	(4,2.5) -- (-0.5,2.5)	node [left] {$H_{y}$} node[pos=0, right] {$y=\infty$};
			\draw [black, line width = 1pt] 	(0,3) -- (0,-0.5)		node [below] {$H_{x}$} node[pos=0, above, xshift=-7pt] {$x=0$};
			\draw [black, line width = 1pt] 	(3,3) -- (3,-0.5)		node [below] {$H_{x}$} node[pos=0, above, xshift=7pt] {$x=\infty$};
			
			\node (q1) at (0,1.3) 	[basept,label={[xshift=-10pt, yshift = -10 pt] $q_{1}$}] {};
			\node (q2) at (3,1.1) 	[basept,label={[xshift=10pt, yshift = -10 pt] $q_{2}$}] {};
			\node (q3) at (3,0) [basept,label={[xshift=10pt, yshift =0 pt] $q_{3}$}] {};
			\node (q4) at (0,2.5) [basept,label={[xshift=-10pt, yshift =-15 pt] $q_{4}$}] {};
			\node (q5) at (0.5,3.0) 	[basept,label={[above ] $q_{5}$}] {};
			\node (q6) at (1.1,3.0) 	[basept,label={[above ] $q_{6}$}] {};
			\node (q7) at (1.7,3.0)		[basept,label={[above ] $q_{7}$}] {};
			\node (q8) at (2.3,3.0) [basept,label={[above ] $q_{8}$}] {};
			\draw [line width = 0.8pt, ->] (q8) edge (q7) (q7) edge (q6) (q6) edge (q5) (q5) edge (q4);
		\end{scope}

		\draw [->] (7,1.5)--(5,1.5) node[pos=0.5, below] {$\operatorname{Bl}_{q_{1}\cdots q_{8}}$};

		\begin{scope}[xshift = 9cm]
			\draw [blue, line width = 1pt] 	(1.5,2.5) -- (4,2.5)	node [right] {$H_{y} - F_{4}-F_{5}$};
			\draw [blue, line width = 1pt] 	(0,2.2) -- (0,-0.5)		node [below] {$H_{x}-F_{1}-F_{4}$};
			\draw [red, line width = 1pt ] (-0.5,0) .. controls (2,0) and (2.5,-0.5) .. (3,-1) node [below] {$H_{y} - F_{3}$};
			\draw [blue, line width = 1pt] 	(3.5,3)  .. controls (3.5,1.5) and (3.7,1) .. (4,0.5)  node [right] {$H_{x} - F_{2}-F_{3}$};

			\draw [red, line width = 1 pt] (-0.4,0.9) -- (0.4,1.7) node [pos = 0, left] {$F_{1}$};
			\draw [red, line width = 1 pt] (3.1,0.7) -- (3.9,1.5) node [right]
 {$F_{2}$};
			\draw [blue, line width = 1 pt] (-0.4,1.5) -- (1.6,3.5)  node [pos = 0, left] {$F_{4}-F_5$};
             \draw [blue, line width = 1 pt] (1.2,2.6) -- (2.6,4)  node [above right]
{$F_{6}-F_7$};
            \draw [blue, line width = 1 pt] (2.4,2.0) -- (0.6,3.6)  node [above left]
{$F_{5}-F_6$};
            \draw [blue, line width = 1 pt] (3.2,2.8) -- (2.0,3.8)  node [above left]
{$F_{7}-F_8$};
			\draw [red, line width = 1 pt] (2,-1) -- (4,1) node [right] {$F_{3}$};
			\draw [red, line width = 1 pt] (2.5,2.8) -- (3.3,3.6) node [right]
 {$F_{8}$};
		\end{scope}
	\end{tikzpicture}
	\caption{The Sakai surface for the semiclassical Laguerre weight recurrence.}
	\label{fig:surface-semi-laguerre}
\end{figure}
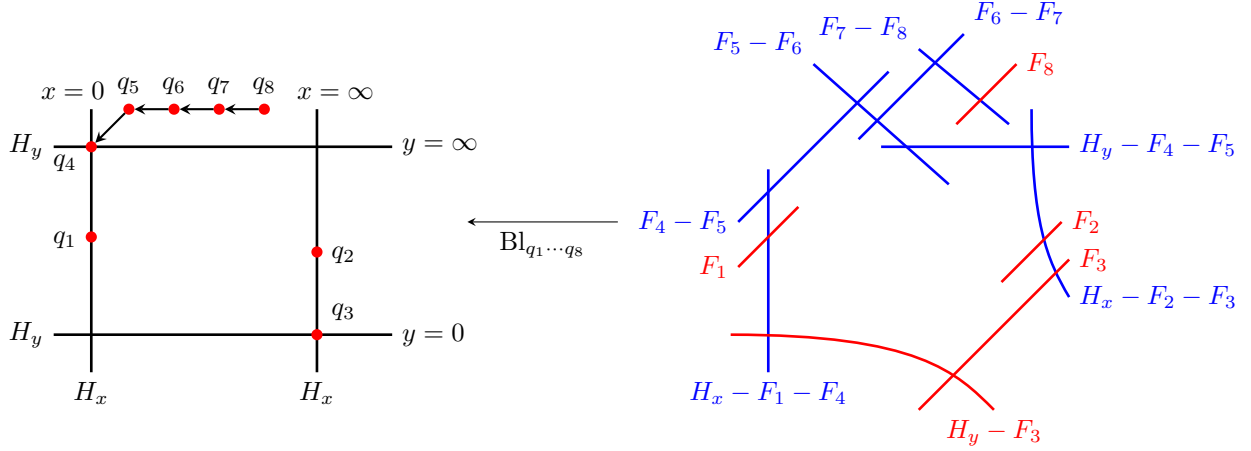

% subsection the_surface_type (end)

\subsection{Initial Geometry Identification} % (fold)
\label{sub:initial_geometry_identification}
The next step in the identification procedure is to determine a basis transformation within $\operatorname{Pic}(\mathcal{X})$ from the basis $\{\mathcal{H}_{x},\mathcal{H}_{y},\mathcal{F}_{i}\}$ to the basis $\{\mathcal{H}_{q},\mathcal{H}_{p},\mathcal{E}_{j}\}$ that corresponds to the standard example. Here, we conduct an initial geometry matching, acknowledging that the resulting basis transformation is non-unique and may later require adjustment to match the dynamics.
\begin{lemma}\label{lem:change-basis-pre}
	The following change of basis of $\operatorname{Pic}(\mathcal{X})$ identifies the root bases
	between the standard $E_{6}^{(1)}$ surface and the surface that we obtained for the semiclassical Laguerre weight recurrence:
	\begin{align*}
\begin{aligned}
&\mathcal{H}_x=\mathcal{H}_p,
\\
&\mathcal{H}_y=\mathcal{H}_q+\mathcal{H}_p-\mathcal{E}_1-\mathcal{E}_3,
\\
&\mathcal{F}_1=\mathcal{E}_4,
\\
&\mathcal{F}_2=\mathcal{E}_2,
\\
&\mathcal{F}_3=\mathcal{H}_p-\mathcal{E}_1,
\\
&\mathcal{F}_4=\mathcal{H}_p-\mathcal{E}_3,
\\
&\mathcal{F}_5=\mathcal{E}_5,
\\
&\mathcal{F}_6=\mathcal{E}_6,
\\
&\mathcal{F}_7=\mathcal{E}_7,
\\
&\mathcal{F}_8=\mathcal{E}_8,
\end{aligned}
&&
\begin{aligned}
&\mathcal{H}_q=\mathcal{H}_x+\mathcal{H}_y-\mathcal{F}_3-\mathcal{F}_4,
\\
&\mathcal{H}_p=\mathcal{H}_x,
\\
&\mathcal{E}_1=\mathcal{H}_x-\mathcal{F}_3,
\\
&\mathcal{E}_2=\mathcal{F}_2,
\\
&\mathcal{E}_3=\mathcal{H}_x-\mathcal{F}_4,
\\
&\mathcal{E}_4=\mathcal{F}_1,
\\
&\mathcal{E}_5=\mathcal{F}_5,
\\
&\mathcal{E}_6=\mathcal{F}_6,
\\
&\mathcal{E}_7=\mathcal{F}_7,
\\
&\mathcal{E}_8=\mathcal{F}_8.
\end{aligned}
\end{align*}
\end{lemma}
\begin{proof}
	This computation is straightforward and relies on comparing the surface root bases illustrated in Figure~\ref{fig:d-roots-slw} and  Figure~\ref{fig:d-roots-e61}.
\end{proof}

% subsection initial_geometry_identification (end)

\subsection{The Symmetry Roots and the Translations} % (fold)
\label{sub:the_symmetry_roots_and_the_translations}
We are now in the position to start comparing the dynamics. By starting with the standard selection of the symmetry root basis \eqref{eq:a-roots-a2} and utilizing the basis transformation in Lemma~\ref{lem:change-basis-pre}, we derive the symmetry roots for the applied problem, as illustrated in Figure \ref{fig:a-roots-slw-pre}.
\begin{figure}[ht]
\begin{equation}\label{eq:a-roots-slw-pre}			
	\raisebox{-32.1pt}{\begin{tikzpicture}[
			elt/.style={circle,draw=black!100,thick, inner sep=0pt,minimum size=2mm}]
		\path 	(0,1) 	node 	(a0) [elt, label={[xshift=-10pt, yshift = -10 pt] $\alpha_{0}$} ] {}
		        (-1.15,-1) node 	(a1) [elt, label={[xshift=-10pt, yshift = -10 pt] $\alpha_{1}$} ] {}
		        (1.15,-1) node  	(a2) [elt, label={[xshift=10pt, yshift = -10 pt] $\alpha_{2}$} ] {};
		\draw [black,line width=1pt ] (a0) -- (a1) -- (a2)  -- (a0);
	\end{tikzpicture}} \qquad
			\begin{aligned}
&\alpha_0=2\mathcal{H}_x+\mathcal{H}_y-\mathcal{F}_3-\mathcal{F}_4-\mathcal{F}_5-\mathcal{F}_6-\mathcal{F}_7-\mathcal{F}_8,
\\
&\alpha_1=\mathcal{H}_y-\mathcal{F}_1-\mathcal{F}_3,\quad
\\
&\alpha_2=\mathcal{F}_3-\mathcal{F}_2,
\\
&\delta=\alpha_0+\alpha_1+\alpha_2.
\end{aligned}
\end{equation}
	\caption{The symmetry root basis for the semiclassical Laguerre weight recurrence (preliminary choice).}
	\label{fig:a-roots-slw-pre}	
\end{figure}

From the action on $\operatorname{Pic}(\mathcal{X})$ specified in Lemma~\ref{lem:dyn}, we immediately deduce that the corresponding translation on the root lattice. We decompose $\psi$ in terms of the generators of the extended affine Weyl symmetry group (see Section~\ref{sub:the_extended_affine_weyl_symmetry_group}) and compare the results with the standard mapping $\varphi$ given in Section~\ref{sub:Standard_discrete_d_dpain_a__2_1_e__6_1_equation}. We obtain
\begin{alignat*}{2}
	\psi_{*}: \upalpha &=  \langle \alpha_{0}, \alpha_{1}, \alpha_{2}  \rangle
	\mapsto \psi_{*}(\upalpha) = \upalpha + \langle 1,-1,0 \rangle \delta,&\qquad
	\psi &= \sigma_{1}\sigma_{2}w_{2}w_{1},\\
\varphi_{*}: \upalpha &=  \langle \alpha_{0}, \alpha_{1}, \alpha_{2}  \rangle
	\mapsto \varphi_{*}(\upalpha) = \upalpha + \langle 0,1,-1 \rangle \delta,&\qquad
	\varphi &= \sigma_{1}\sigma_{2}w_{0}w_{2}.
\end{alignat*}

We immediately observe that $\psi = w_{1}\circ \varphi \circ w_{1}^{-1}$ (note that $w_{1} \sigma_{1}\sigma_{2}  = \sigma_{1}\sigma_{2} w_{0}$
and that $w_{1}$ is an involution, $w_{1}^{-1} = w_{1}$). Consequently, our dynamics are indeed equivalent to the standard equation~\eqref{eq:dPE6-KNY}, but the change of basis in Lemma~\ref{lem:change-basis-pre} needs to be adjusted by acting by $w_1$. We do it in the next section.

% subsection the_symmetry_roots_and_the_translations (end)

\subsection{Final Geometry Identification} % (fold)
\label{sub:final_geometry_identification}
\begin{lemma}\label{lem:change-basis-fin}
	After the change of basis of $\operatorname{Pic}(\mathcal{X})$ given by
	\begin{align*}
\begin{aligned}
&\mathcal{H}_x=\mathcal{H}_q+\mathcal{H}_p-\mathcal{E}_3-\mathcal{E}_4,
\\
&\mathcal{H}_y=\mathcal{H}_q+\mathcal{H}_p-\mathcal{E}_1-\mathcal{E}_3,
\\
&\mathcal{F}_1=\mathcal{H}_q-\mathcal{E}_3,
\\
&\mathcal{F}_2=\mathcal{E}_2,
\\
&\mathcal{F}_3=\mathcal{H}_q+\mathcal{H}_p-\mathcal{E}_1-\mathcal{E}_3-\mathcal{E}_4,
\\
&\mathcal{F}_4=\mathcal{H}_p-\mathcal{E}_3,
\\
&\mathcal{F}_5=\mathcal{E}_5,
\\
&\mathcal{F}_6=\mathcal{E}_6,
\\
&\mathcal{F}_7=\mathcal{E}_7,
\\
&\mathcal{F}_8=\mathcal{E}_8,
\end{aligned}
&&
\begin{aligned}
&\mathcal{H}_q=\mathcal{H}_x+\mathcal{H}_y-\mathcal{F}_3-\mathcal{F}_4,
\\
&\mathcal{H}_p=\mathcal{H}_x+\mathcal{H}_y-\mathcal{F}_1-\mathcal{F}_3,
\\
&\mathcal{E}_1=\mathcal{H}_x-\mathcal{F}_3,
\\
&\mathcal{E}_2=\mathcal{F}_2,
\\
&\mathcal{E}_3=\mathcal{H}_x+\mathcal{H}_y-\mathcal{F}_1-\mathcal{F}_3-\mathcal{F}_4,
\\
&\mathcal{E}_4=\mathcal{H}_y-\mathcal{F}_3,
\\
&\mathcal{E}_5=\mathcal{F}_5,
\\
&\mathcal{E}_6=\mathcal{F}_6,
\\
&\mathcal{E}_7=\mathcal{F}_7,
\\
&\mathcal{E}_8=\mathcal{F}_8.
\end{aligned}
\end{align*}
		the recurrence relations \eqref{eq:xyn-evol}
		for variables $x_{n}$ and $y_{n}$ coincides with the discrete Painlev\'e equation given by
		\eqref{eq:dPE6-KNY}. The resulting identification of the symmetry root bases (the surface root bases do not change) is shown in
		Figure~\ref{fig:a-roots-slw-fin}.
\end{lemma}
\begin{figure}[ht]
\begin{equation*}\label{eq:a-roots-slw-fin}			
	\raisebox{-32.1pt}{\begin{tikzpicture}[
			elt/.style={circle,draw=black!100,thick, inner sep=0pt,minimum size=2mm}]
		\path 	(0,1) 	node 	(a0) [elt, label={[xshift=-10pt, yshift = -10 pt] $\alpha_{0}$} ] {}
		        (-1.15,-1) node 	(a1) [elt, label={[xshift=-10pt, yshift = -10 pt] $\alpha_{1}$} ] {}
		        (1.15,-1) node  	(a2) [elt, label={[xshift=10pt, yshift = -10 pt] $\alpha_{2}$} ] {};
		\draw [black,line width=1pt ] (a0) -- (a1) -- (a2)  -- (a0);
	\end{tikzpicture}} \qquad
			\begin{aligned}
&\alpha_0=2\mathcal{H}_x+2\mathcal{H}_y-\mathcal{F}_1-2\mathcal{F}_3-\mathcal{F}_4-\mathcal{F}_5-\mathcal{F}_6-\mathcal{F}_7-\mathcal{F}_8,
\\
&\alpha_1=-\mathcal{H}_y+\mathcal{F}_1+\mathcal{F}_3,\quad
\\
&\alpha_2=\mathcal{H}_y-\mathcal{F}_1-\mathcal{F}_2,
\\
&\delta=\alpha_0+\alpha_1+\alpha_2.
\end{aligned}
\end{equation*}
	\caption{The symmetry root basis for the semiclassical Laguerre weight recurrence (final choice).}
	\label{fig:a-roots-slw-fin}	
\end{figure}
Next we need to realize this change of basis on $\operatorname{Pic}(\mathcal{X})$ by an explicit change of coordinates. For that, it is convenient to first
match the parameters between the applied problem and the standard example. This is done with the help of the \emph{Period Map}.

% subsection final_geometry_identification (end)

\subsection{The Period Map and the Identification of Parameters} % (fold)
\label{sub:the_period_map_and_the_identification_of_parameters}
Before establishing the coordinate transformation that identifies the two dynamics, we need to match the semiclassical Laguerre weight parameters $\lambda$, $s$ and the recurrence step $n$ with the root variables $a_i$. These root variables are defined through the \emph{Period Map} $\mathcal{X}:Q\rightarrow\mathbb{C}$, where $a_{i}=\mathcal{X}(\alpha_{i})$. Additionally, it is readily observed that the points $q_i$ lie on the polar divisor of a symplectic form, which in the affine $(x,y)$-chart is given by $\omega = k \frac{dx\wedge dy}{x}$. Consequently, the computation of the period map relies on the following results from \cite{Sak:2001:RSAWARSGPE}:
\begin{enumerate}[$\bullet$]
		\item Each symmetry root $\alpha_{i}$ can be represented (non-uniquely) as the difference of classes of two effective divisors, $\alpha_{i}=[C_{i}^1]-[C_{i}^0]$.
		
		\item For each such representation, there exists a unique irreducible component $d_k$ of the anti-canonical divisor $-K_{\mathcal{X}}$ satisfying $d_k\bullet C_{i}^1=d_k\bullet C_{i}^0=1$. Define the intersection points $P_i:=d_k\cap C_i^0$ and $Q_i:=d_k\cap C_i^1$.

        \item Consequently, the period map $\mathcal{X}$ acts on $\alpha_i$ as
        \begin{equation}
        \mathcal{X}(\alpha_i)=\mathcal{X}([C_{i}^1]-[C_{i}^0])=\int_{P_i}^{Q_i}\frac{1}{2\pi \mathrm{i}}\oint_{d_k}\omega=\int_{P_i}^{Q_i}\mathrm{res}_{d_k}\omega.
        \end{equation}
		\end{enumerate}
Building on this, we establish the following lemma.
\begin{lemma}
	\qquad
	
	\begin{enumerate}[(i)]
		\item The residues of the symplectic form $\omega = k \frac{\mathrm{d}x\wedge \mathrm{d}y}{x}$
		along the irreducible components of the polar divisor are given by
		\begin{alignat*}{4}
			\operatorname{res}_{d_{0}} \omega &=  -\frac{k}{4}\mathrm{d}v_7, &\qquad
			\operatorname{res}_{d_{1}} \omega &=  -k\mathrm{d}y,\quad &\qquad
			\operatorname{res}_{d_{2}} \omega &=  0,\quad &\qquad
			\operatorname{res}_{d_{3}} \omega &=  0,\\
			\operatorname{res}_{d_{4}} \omega &=  0, &\qquad
            \operatorname{res}_{d_{5}} \omega &=  k\mathrm{d}y, &\qquad
			\operatorname{res}_{d_{6}} \omega &=  k\frac{v_6}{4}\mathrm{d}v_6.
		\end{alignat*}
		
		\item For the standard root variable normalization $\mathcal{X}(\delta)=a_{0}+a_{1}+a_{2}=1$ we need to take $k=-1$ and root variables $a_{i}$ are then given by

		\begin{equation}\label{eq:root-vars-slw}
			a_{0} =  1-\lambda, \qquad a_{1} = -n,\qquad
			a_{2} =  n + \lambda.
		\end{equation}
		\end{enumerate}
\end{lemma}
\begin{proof}
For comprehensive examples of such calculations, refer to \cite{DzhTak:2018:OSAOSGTODPE,DzhFilSto:2020:RCDOPHWDPE}, here, we only explain one example. Consider the root $\alpha_0$ and represent it as a difference of two effective classes,
\begin{align*}
\alpha_0&=2\mathcal{H}_x+2\mathcal{H}_y-\mathcal{F}_1-2\mathcal{F}_3-\mathcal{F}_4-\mathcal{F}_5-\mathcal{F}_6-\mathcal{F}_7-\mathcal{F}_8\\
&=[2H_{x}+2H_{y}-F_{1}-F_{3}-F_{4}-F_{5}-F_{6}-F_{7}-F_{8}]-[F_{3}].
\end{align*}
The first class is a class of a proper transform of a $(2,2)$-curve passing through points $q_{1},q_{3},\ldots,q_{8}$. A direct computation reveals that its equation in the affine $(x,y)$ chart is $c_{1}(y-n)+x(c_{2}+c_{1}y(t-2x(y+\lambda-1)))=0$, where $c_1$ and $c_2$ are parameters. The second class corresponds to the exceptional divisor $F_3$. Subsequently, we utilize the irreducible component $d_1$, leading to
\begin{equation*}
(H_{x}-F_{2}-F_{3})\bullet(2H_{x}+2H_{y}-F_{1}-F_{3}-F_{4}-F_{5}-F_{6}-F_{7}-F_{8})=(H_{x}-F_{2}-F_{3})\bullet F_{3}=1,
\end{equation*}
so we need to consider these curves in the $(X,y)$-chart. In this chart, the proper transform $2H_{x}+2H_{y}-F_{1345678}$ intersects $d_1$ at $(X=0,1-\lambda)$, while the exceptional divisor $F_3$ intersects $d_1$ at the point $q_3(X=0,0)$. Computing the symplectic form $\omega$ in the $(X,y)$-chart,
\begin{equation*}
\omega = k \frac{\mathrm{d}x\wedge \mathrm{d}y}{x} = -k \frac{\mathrm{d}X\wedge \mathrm{d}y}{X},
\end{equation*}
we see that
\begin{equation*}
\operatorname{res}_{d_{1}}\omega=\operatorname{res}_{X=0}-k \frac{\mathrm{d}X\wedge \mathrm{d}y}{X}=-k\mathrm{d}y, \qquad a_{0}=\mathcal{X}(\alpha_0)=\int_{0}^{1-\lambda}-kdy=k(\lambda-1).
\end{equation*}
Similarly, we obtain
\begin{equation*}
a_{0}=k(\lambda-1), \qquad a_{1}=kn, \qquad a_{2}=-k(n+\lambda).
\end{equation*}
Imposing the normalization condition $\mathcal{X}(\delta)=a_{0}+a_{1}+a_{2}=-k=1$, we find that $k=-1$.
\begin{remark}
Note that the root variable evolution under the discrete step $n\rightarrow n+1$ is given
by
\begin{equation*}
\overline{a}_{0} = a_{0}, \qquad \overline{a}_{1} = a_{1}-1, \qquad \overline{a}_{2} = a_{2} + 1,
\end{equation*}
which corresponds to the standard translation on the root basis given by \eqref{eq:dPE6-rv-evol}.
\end{remark}
\end{proof}

% subsection the_period_map_and_the_identification_of_parameters (end)

\subsection{The Change of Coordinates} % (fold)
\label{sub:the_change_of_coordinates}
We are now ready to prove Theorem~\ref{thm:coordinate-change}, which is the main result of the paper.
\begin{proof}
\textbf{(Theorem~\ref{thm:coordinate-change})} This computation is standard, with detailed examples available in \cite{DzhTak:2018:OSAOSGTODPE,DzhFilSto:2020:RCDOPHWDPE}, so we only provide a brief outline here. From the change of basis in Lemma~\ref{lem:change-basis-fin} on the Picard lattice for the coordinate classes,
\begin{equation*}
\mathcal{H}_x=\mathcal{H}_q+\mathcal{H}_p-\mathcal{E}_3-\mathcal{E}_4, \quad
\mathcal{H}_y=\mathcal{H}_q+\mathcal{H}_p-\mathcal{E}_1-\mathcal{E}_3,
\end{equation*}
we see that $x$ and $y$ are projective coordinates on pencils of $(1,1)$-curves in the $(q,p)$-coordinates, and $x$ corresponds to the pencil passing through $p_{3}$ and $p_{4}$, while $y$ corresponds to the pencil passing through $p_{1}$, $p_{3}$. Thus, we take the change of coordinates to be
\begin{equation*}
x(q,p)=\frac{A(qp-a_1)+Bq}{C(qp-a_1)+Dq}, \quad
y(q,p)=\frac{Kqp+L}{Mqp+N},
\end{equation*}
where the coefficients $A,\ldots,N$ are still to be determined. For example, the correspondence $H_{p} - E_{3} - E_{5} = F_{4} - F_{5}$ means that
\begin{equation*}
(x,Y)(q,P=0)=(\frac{Aq+Bq\cdot0-Aa_{1}\cdot0}{Cq+Dq\cdot0-Ca_{1}\cdot0},\frac{N\cdot0+Mq}{L\cdot0+Kq})=(\frac{A}{C},\frac{M}{K})=(0,0),\quad \text{and so}\quad A=M=0,
\end{equation*}
then we can take $B=N=1$ to get
\begin{equation*}
x(q,p)=\frac{q}{C(qp-a_1)+Dq}, \quad
y(q,p)=Kqp+L.
\end{equation*}
The correspondence $E_{1} - E_{2} = H_{x} -F_{2} - F_{3}$ means that $X(Q=0,0)=D=0$, thus,
\begin{equation*}
x(q,p)=\frac{q}{C(qp-a_1)}.
\end{equation*}
Proceeding in the same way, from the correspondence $E_{6} - E_{7} = F_{6} - F_{7}$ we deduce that $C^2=2K$, from the correspondence $E_{4} = H_{y} - F_{3}$ we get $L=-\frac{a_1C^2}{2}$. Thus, we get
\begin{equation*}
y(q,p)=\frac{C^2}{2}(qp-a_1).
\end{equation*}
Finally, from the correspondence $E_{7} - E_{8} = F_{7} - F_{8}$ we get $Ct=-s$, from the correspondence $E_{8} = F_{8}$ and \eqref{eq:xyn-evol}, we obtain that $C=-\sqrt{2}$.  The
inverse change of variables can now be either obtained directly, or computed in the similar
way. This concludes the proof of Theorem~\ref{thm:coordinate-change}.
\end{proof}

% subsection the_change_of_coordinates (end)

% section the_identification_procedure (end)
\section*{Acknowledgements} % (fold)
\label{sec:acknowledgements}
We express our sincere thanks to Prof. A. Dzhamay for his enthusiastic help and valuable discussions.

M. Zhu acknowledges the support of the National Natural Science Foundation of China under Grant No. 12201333, the Natural Science Foundation of Shandong Province (Grant No. ZR2021QA034), and the Breeding Plan of Shandong Provincial Qingchuang Research Team (Grant No. 2023KJ135).

%section acknowledgements (end)

% \small
% \bibliographystyle{amsxport}
\bibliographystyle{amsalpha}

\providecommand{\bysame}{\leavevmode\hbox to3em{\hrulefill}\thinspace}
\providecommand{\MR}{\relax\ifhmode\unskip\space\fi MR }
\newcommand{\etalchar}[1]{$^{#1}$}
\providecommand{\href}[2]{#2}

\end{document}